\documentclass{elsarticle}

\usepackage{amsfonts}
\usepackage{amsthm}

\usepackage{amssymb}
\usepackage{mathtools} 
\usepackage{xcolor}
\usepackage{soul} 
\usepackage{mathrsfs}
\usepackage{tabularx}
\usepackage[top=1in, bottom=1in, left=1in, right=1in]{geometry}

\newtheorem{theorem}{Theorem}[section]
\newtheorem{lemma}{Lemma}[section]

\theoremstyle{definition}
\newtheorem{definition}{Definition}[section]

\theoremstyle{remark}

\newtheorem{example}{Example}[section]

\DeclareMathOperator*{\esssup}{ess\,sup}

\numberwithin{equation}{section}

\begin{document}
	
	\begin{frontmatter}
		
		\title{Necessary first and second order optimality conditions for a fractional order differential equation with state delay}
		
		

			\author[]{Jasarat J. Gasimov}
		\ead{jasarat.gasimov@emu.edu.tr}
		\author[]{Nazim I. Mahmudov}
		\ead{nazim.mahmudov@emu.edu.tr}
		\cortext[cor1]{Corresponding author}
		\cortext[cor1]{Corresponding author}	
		
		\address{Department of Mathematics, Eastern Mediterranean University, Mersin 10, 99628, T.R. North Cyprus, Turkey}

		
		\begin{abstract}
			\noindent
			 In this research paper, we examine an optimal control problem involving a dynamical system governed by
			a nonlinear Caputo fractional time-delay state equation. The primary objective of this study is to obtain
			the necessary conditions for optimality, both the first and second order, for the Caputo fractional time-delay
			optimal control problem. We derive the first-order necessary condition for optimality for the given fractional
			time-delay optimal control problem. Moreover, we focus on a case where the Pontryagin maximum principle
			degenerates, meaning that it is satisfied in a tivial manner. Consequently, we proceed to derive the second
			order optimality conditions specific to the problem under investigation. At the end illustrative examples are
			provided.
		\end{abstract}

		\begin{keyword}
			retarded differential equation; a constant delay; pontryagin maximum principle; optimal control
		\end{keyword}  
		
	\end{frontmatter}
	\section{Introduction}\label{Sec:intro}
	Fractional differential equations have attracted significant attention in the last 30 years \cite{2,3,4,5,6,7,16,22,24,29,35,36,37}. This increased interest can be attributed to the advancements in fractional calculus theory and its diverse applications. One area of active research involves studying optimal control problems for systems incorporating fractional Riemann-Liouville and Caputo derivatives. These problems have wide-ranging applications in various fields such as biology [32], chemistry [9], economics [14], electrical engineering [17], and medicine [18]. It is worth noting that fractional optimal control problems, which are described by ordinary fractional differential equations, can be considered as an extension of traditional optimal control problems.
	
	Fractional optimal control problems, which can be represented by ordinary fractional differential equations, can be considered as an advancement or an expansion of conventional optimal control problems.
	
	The Pontryagin maximum principle is a key outcome in the theory of essential optimality conditions at the first order. It was originally demonstrated in [21] for optimal control problems involving ordinary differential equations. Subsequently, different essential optimality conditions have been derived for a range of systems, encompassing both first and higher order conditions [10, 21, 20, 33, 34].
	
	An admissible control is referred to as a Pontryagin extremal if it satisfies the Pontryagin maximum condition. The maximum principle implies that any optimal control is a Pontryagin extremal. Thus, the problem of finding the optimal control is essentially reduced to selecting the best control among the Pontryagin extremals. If there exists an optimal control in the optimization problem and the Pontryagin extremal is unique, then it is the optimal control. However, it is important to note that even in simple cases, non-optimal controls can satisfy the Pontryagin maximum principle. In such situations, the maximum principle cannot effectively eliminate all non-optimal controls.
	
	The process of narrowing down a subset of Pontryagin extremals that can be deemed optimal is a crucial and challenging task. Its complexity stems from the fact that the maximum principle itself serves as a very strong necessary condition for optimality.
	
	Consequently, investigating scenarios where the Pontryagin maximum principle fails to provide a unique determination of the optimal control has become a key focus in the advancement of optimal systems theory. Therefore, the need arises to develop new necessary conditions for optimality that can narrow down the set of controls that satisfy the Pontryagin maximum principle.
	
	The paper [38] addresses the minimization problem of an integral functional with a non-convex control-dependent integrand. The focus is on solutions of a Riemann-Liouville fractional differential control system with mixed non-convex feedback constraints on the control. In [39], optimal control problems of fractional order are examined, where the dynamic control system involves derivatives of both integer and fractional orders, and the final time is not fixed. Necessary optimality conditions are derived for the triplet of state/control/terminal time. The cases with finite time constraints are also considered. It is demonstrated that, under suitable assumptions, the derived necessary optimality conditions become sufficient. In the paper [35], a necessary first-order optimality condition is obtained in the form of the Pontryagin maximum principle. Additionally, in the case of degenerate maximum principle, a necessary optimality condition is derived for singular controls according to the Pontryagin maximum principle. It should be noted that these results are obtained using the fundamental matrix. If the Pontryagin maximum principle holds along the control trajectory without degeneracy, then the findings of [8, 35] consider this control as a potential candidate for optimality. Hence, there is a need for new necessary optimality conditions that can effectively single out this control from the set of candidates for optimality.
	
	This paper presents one of the potential approaches for establishing higher-order necessary optimality conditions for fractional optimal control problems based on the Caputo sense.
	
	We investigate an optimal control problem for a dynamic system governed by a nonlinear differential equation, incorporating the Caputo fractional derivative of order $\alpha \in(0, 1). $The control process occurs within a fixed and finite time interval. The objective is to minimize a Bolza-type cost functional consisting of two components. The first component assesses the state vector of the system at a predetermined terminal time T, while the second component involves the integral evaluation of a control over the entire time interval [0, T]. Notably, the second term corresponds to the Riemann-Liouville integral of order $\beta \geq \alpha.$ The set of control functions assumes nonempty and compact values.
	
	The presented optimal control problem is investigated using a novel variation of the increment method, which extensively employs the concept of an adjoint integral equation. It is worth noting that the solution of this integral equation does not exhibit any singularity.

	\section{Preliminaries}
In this section, we will begin by providing concise explanations of key terms derived from fractional calculus that hold significance in this research.

	Let $R^n$ and $R^{n \times n}$ denote the spaces of $n$-dimensional vectors and $n \times n$ matrices, respectively, and let $E \in R^{n \times n}$ be the identity matrix. A norm is defined in $R^n$ by $\Vert \cdot \Vert$, and in $R^{n \times n}$ by an appropriate norm for matrices. Consider the intervals $[0,T] \subset R$ and $[-h,T] \subset R$, where $T > 0$ and $h > 0$, and Let $Y$ be either $R^n$ or $R^{n \times n}$.
	
	In the space $L^\infty([0,T],Y)$, we have a set of Lebesgue measurable functions $\varphi: [0,T] \rightarrow Y$ such that $\Vert \varphi (\cdot) \Vert_{[0,T]} < \infty$. Here, $\Vert \varphi (\cdot) \Vert_{[0,T]}$ is defined as the essential supremum of the norm of $\varphi(t)$ over all $t \in [0,T]$. Essentially, it measures the largest possible norm value attained by $\varphi(t)$ in that interval.
	
	Furthermore, in the space $C([0,T],Y)$, we have a set of functions $\varphi: [0,T] \rightarrow Y$ such that $\Vert \varphi (\cdot) \Vert = \max_{t \in [0,T]} \Vert \varphi(t) \Vert$. In other words, it represents the maximum norm value achieved by $\varphi(t)$ over the interval $[0,T]$.
	
	We also have the space of summable functions denoted as $L^1([0,T],Y)$, which encompasses functions that are integrable over the interval $[0,T]$ with respect to the appropriate norm in $Y$.

	\begin{definition}(\cite{29})
		The frectional integral of order $\alpha>0$ for a function $\varphi:[0,T]\rightarrow Y$ is defined by 
		\begin{equation*}
			(I^{\alpha}_{0+}\varphi)(t)=\frac{1}{\Gamma (\alpha)}\int_{0}^{t}(t-\tau)^{\alpha-1}\varphi(\tau)d\tau, \quad t>0,
		\end{equation*}
	where $\Gamma(\cdot)$ is the well-known Euler's gamma function.
	\end{definition}
\begin{definition}(\cite{29})
	The Riemann-Liouville frectional derivative of order $0<\alpha\leq1$ for a function $\varphi(\cdot)\in L^{\infty}([0,T],Y)$ is defined by 
	\begin{equation*}
		( ^{RL}D^{\alpha}_{0+}\varphi)(t)=\frac{1}{\Gamma (1-\alpha)}\frac{d}{dt}\int_{0}^{t}(t-\tau)^{-\alpha}\varphi(\tau)d\tau, \quad t>0,
	\end{equation*}
\end{definition}
\begin{definition}(\cite{29})
	The Caputo frectional derivative of order $0<\alpha\leq1$ for a function $\varphi(\cdot)\in L^{\infty}([0,T],Y)$ is defined by 
	\begin{equation*}
		( ^{C}D^{\alpha}_{0+}\varphi)(t)=\frac{1}{\Gamma (1-\alpha)}\int_{0}^{t}(t-\tau)^{-\alpha}\frac{d}{d\tau}\varphi(\tau)d\tau, \quad t>0,
	\end{equation*}
\end{definition}
The relationship between the Caputo and Riemann-Liouville fractional differentiation operators for the function $\varphi(\cdot)\in L^{\infty}([0,T],Y)$ can be described as follows:
\begin{equation*}
^{C}D^{\alpha}_{0+}\varphi(t)= ^{RL}D^{\alpha}_{0+}\big(\varphi(t)-\varphi(0)\big), \quad \alpha\in (0,1].
\end{equation*}
We denote
$$
AC^{\alpha}_{\infty}([0,T],Y)=\left\{y(\cdot)| y:[0,T]\rightarrow Y, y(t)=y(0)+(I^{\alpha}_{0+}\varphi)(t), \quad t\in[0,T]\right\}
$$
\begin{lemma}(\cite{29})
For any $y(t)\in AC^{\alpha}_{\infty}([0,T],Y),$ the value $( ^{C}D^{\alpha}_{0+}y)(t)$ is correctly defined for a.e. $t\in [0,T].$ Furthermore, the inclusion $( ^{C}D^{\alpha}_{0+}y)(\cdot)\in L^{\infty}([0,T],Y)$ holds (i.e., there exists $\varphi(\cdot)\in L^{\infty}([0,T],Y)$ such that $\varphi(t)=( ^{C}D^{\alpha}_{0+}y)(t)$ for a.e. $t\in[0,T]$) and 
$$
(I^{\alpha}_{0+}( ^{C}D^{\alpha}_{0+}y)(t))(t)=y(t)-y(0), \quad t\in[0,T].
$$
\end{lemma}
\begin{lemma}(\cite{29})
	If $\alpha\in (0,1)$ and $\varphi(\cdot)\in L^{1}([0,T],Y),$ then $(D^{\alpha}_{0+}( ^{C}I^{\alpha}_{0+}y)(t))(t)=\varphi(t),$ $\quad$ a.e. $t\in[0,T].$
\end{lemma}
	\begin{lemma}\label{Lmm}(\cite{35})
		Let $Y$ be Banach space and suppose $x\in C([-h,T],Y)$ satisfies the following inequality 
		$$\Vert y(t)\Vert\leq a(t)+b\int_{0}^{t}(t-s)^{(\beta-1)}\Vert y(s)\Vert ds+c\int_{0}^{t}(t-s)^{(\beta-1)}\Vert y_{h}\Vert_{C} ds \quad 0\leq t\leq T, h>0$$
		$$y(t)=\psi(t) \quad -h\leq t\leq0$$
		$\psi\in C, 0<\beta\leq1 and constants, a,b,c\geq0.$ Then there exists constant $M>0$ (indepent on $a$ and $\psi$) such that 
		$$\Vert y(t)\Vert\leq M(a(t)+\Vert\psi\Vert_{C}) \quad 0\leq t\leq T.$$
	\end{lemma}
Note:$\quad\Vert y_{h}\Vert_{C}=\sup_{0\leq\theta\leq h}\Vert y(t-\theta)\Vert.$
\begin{lemma}\label{Lmm1}(\cite{37})
	For an arbitrary function $a(\cdot)\in L^{\infty}(0,T),$ any of its Lebesgue points $\theta \in(0,T),$ and any numbers $\alpha,\varepsilon,\quad 0<\alpha<1, \quad 0<\varepsilon<T-\varepsilon,$ the equality 
	\begin{align*}
		\int_{\theta}^{\theta+\varepsilon}(T-t)^{\alpha-1}(t-\theta)^{\alpha}a(t)dt=\frac{(T-\theta)^{\alpha-1}a(\theta)}{\alpha+1}\varepsilon^{\alpha+1}+o(\varepsilon^{\alpha+1}),
	\end{align*}
holds, where $\lim_{\varepsilon\to0}\frac{o(\varepsilon^{\alpha+1})}{\varepsilon^{\alpha+1}}=0.$
\end{lemma}
\begin{lemma}\label{Lmm2}(\cite{37})
	For an arbitrary function $a(\cdot)\in L^{\infty}(0,T),$ any of its Lebesgue point $\theta \in(0,T)$ of this function, the equality 
	\begin{align*}
	\lim_{t\to\theta;t\in(\theta,T)}\frac{1}{(t-\theta)^{\alpha}} \int_{\theta}^{t}(t-\tau)^{\alpha-1}(t-\theta)^{\alpha-1}\vert a(\tau)-a(\theta)\vert d\tau0, \quad 0<\alpha\leq 1,
	\end{align*}
	holds.
\end{lemma}
\bigskip
	\section{Representation formula}
Let us examine the subsequent nonhomogeneous Cauchy problem, which involves a linear delayed differential equation with variable coefficients and incorporates the Caputo fractional derivatives.
	\begin{equation}\label{cp}
		\begin{cases}
			^c D^{\alpha}_{0+}y(t)=a(t)y(t)+b(t)y(t-\tau)+f(t), \quad a.e. t\in [0,T]\\
			y(0)=y_{0}, \\
			y(t)=\varphi(t);\quad  -h\leq t<0, \quad h>0.
		\end{cases}
	\end{equation}
   $\quad$ where $\quad a(\cdot) , b(\cdot) \in L^{\infty}\bigg([0,T],R^{n\times n}\bigg), f(\cdot) \in L^{\infty}\bigg([0,T],R^{n}\bigg)$, $\quad 0<\alpha<1.\\$
  $\quad$ This  problem has a unique solution $y(\cdot) \in AC^{\alpha}_{\infty}\bigg([0,T],R^{n}\bigg).\\$
  \bigskip
  \begin{lemma}
  	The solution of (\ref{cp}) has the following representation:
  	\begin{align*}
  		y(t)=\frac{1}{\Gamma(\alpha)} \int_{0}^{t} (t-\tau)^{\alpha-1} F(t,\tau)f(\tau) d\tau+ F_{1}(t)y_{0},
  	\end{align*}
  where
  \begin{align*}
  	&F(t,\tau)=E+\frac{(t-\tau)^{1-\alpha}}{\Gamma(\alpha)} \int_{\tau}^{t}(t-s)^{\alpha-1}(s-\tau)^{\alpha-1}F(t,s)a(s)ds\nonumber\\
  	+&\frac{(t-\tau)^{1-\alpha}}{\Gamma(\alpha)} \int_{\tau+h}^{t}(t-s)^{\alpha-1}(s-\tau-h)^{\alpha-1}F(t,s)b(s)ds.
  \end{align*}
  \end{lemma}
\begin{proof}
By performing the operation of multiplying both sides of equation (\ref{cp}) by $\frac{1}{\Gamma(\alpha)}(t-\tau)^{\alpha-1} F(\tau)$ and $F^1$ separately, and subsequently combining the resulting equations, we can readily deduce the following outcome.
\begin{align}\label{cp1}
	&\frac{1}{\Gamma(\alpha)} \int_{0}^{t} (t-\tau)^{\alpha-1} F(\tau) \Bigg[(^c D^{\alpha}_{0+}y)(\tau)-a(\tau)y(\tau)-b(\tau) y(\tau-h)\Bigg]d\tau+F_{1}y(0)\nonumber\\
	=&\frac{1}{\Gamma(\alpha)} \int_{0}^{t} (t-\tau)^{\alpha-1} F(\tau) f(\tau)d\tau+F_{1}y_{0}
\end{align}
Here, $F(t)$ remains an arbitrary $n\times n$ matrix-valued function that is integrable over the interval $[0,T],$ while $F_{1}$ is an arbitrary constant $n\times n$ matrix, $t \in (0,T]$ is a fixed point.$\\$
For a given representation
\begin{equation}\label{cp2}
	y(t)=y(0)+\frac{1}{\Gamma(\alpha)} \int_{0}^{t} (t-\tau)^{\alpha-1}(^c D^{\alpha}_{0+}y)(\tau)d\tau,
\end{equation}
We modify the left-hand side of equation (\ref{cp1}) using the following transformation.
\begin{align}\label{cp3}
	&\frac{1}{\Gamma(\alpha)} \int_{0}^{t} (t-\tau)^{\alpha-1} F(\tau) \Bigg[(^c D^{\alpha}_{0+}y)(\tau)-a(\tau)\bigg(y(0)+\frac{1}{\Gamma(\alpha)} \int_{0}^{\tau} (\tau-s)^{\alpha-1}(^c D^{\alpha}_{0+}y)(s)ds\bigg)\nonumber\\
	-&b(\tau)\biggl(\frac{1}{\Gamma(\alpha)} \int_{0}^{\tau-h} (\tau-h-s)^{\alpha-1}(^c D^{\alpha}_{0+}y)(s)ds+y(0)\biggr)\Bigg]d\tau+F_{1}y(0)\nonumber\\
	=&\frac{1}{\Gamma(\alpha)} \int_{0}^{t} (t-\tau)^{\alpha-1} \Bigg[F(\tau)-\frac{ (t-\tau)^{1-\alpha}}{\Gamma(\alpha)} \int_{\tau}^{t}  (t-s)^{\alpha-1}  (s-\tau)^{\alpha-1} F(s)a(s)ds\\
	-&\frac{(t-\tau)^{1-\alpha}}{\Gamma(\alpha)} \int_{\tau+h}^{t}(t-s)^{\alpha-1}(s-\tau-h)^{\alpha-1}F(s)b(s)ds\Bigg](^c D^{\alpha}_{0+}y)(\tau)d\tau\nonumber\\
	+&\bigg[F_{1}-\frac{1}{\Gamma(\alpha)} \int_{0}^{t} (t-\tau)^{\alpha-1} F(\tau)(a(\tau)+b(\tau))\bigg]y(0)=\frac{1}{\Gamma(\alpha)} \int_{0}^{t} (t-\tau)^{\alpha-1} F(\tau) f(\tau)d\tau+F_{1}y_{0}.\nonumber
\end{align}
In order to simplify this equation, it is necessary for the matrix-valued function $F(t)$ to satisfy the given equation
\begin{align}\label{cp4}
	&F(\tau)=E+\frac{(t-\tau)^{1-\alpha}}{\Gamma(\alpha)} \int_{\tau}^{t}(t-s)^{\alpha-1}(s-\tau)^{\alpha-1}F(s)a(s)ds\nonumber\\
	+&\frac{(t-\tau)^{1-\alpha}}{\Gamma(\alpha)} \int_{\tau+h}^{t}(t-s)^{\alpha-1}(s-\tau-h)^{\alpha-1}F(s)b(s)ds
\end{align}
and the matrix $F_{1}$ has a representation
\begin{equation}\label{cp5}
F_{1}=E+\frac{1}{\Gamma(\alpha)}\int_{0}^{t} (t-\tau)^{\alpha-1} F(\tau)(a(\tau)+b(\tau))d\tau
\end{equation}
where $E \in R^{n\times n}$ is a unit matrix.$\\$
Considering the representation (\ref{cp2}), equation (\ref{cp3}) can be expressed in the following form.

\begin{align}\label{cp6}
y(t)=\frac{1}{\Gamma(\alpha)}\int_{0}^{t} (t-\tau)^{\alpha-1} F(\tau) f(\tau)d\tau+F_{1}y(0)
\end{align}
It is crucial to acknowledge that the matrix-function $F(\tau)$ and the matrix $F_1$ are dependent on the variable $t$ as well. Hence, we express them as $F(t,\tau)$ and $F_1(t)$, respectively, with $t$ belonging to the interval $[0,T]$. Consequently, equations (\ref{cp4}) and (\ref{cp5}) can be rewritten as follows
\begin{align*}
&F(t,\tau)=E+\frac{(t-\tau)^{1-\alpha}}{\Gamma(\alpha)} \int_{\tau}^{t}(t-s)^{\alpha-1}(s-\tau)^{\alpha-1}F(t,s)a(s)ds\nonumber\\
+&\frac{(t-\tau)^{1-\alpha}}{\Gamma(\alpha)} \int_{\tau+h}^{t}(t-s)^{\alpha-1}(s-\tau-h)^{\alpha-1}F(t,s)b(s)ds
\end{align*}
and
\begin{align*}
	F_{1}(t)=E+\frac{1}{\Gamma(\alpha)} \int_{0}^{t}(t-\tau)^{\alpha-1} F(t,\tau)[a(\tau)+b(\tau)]d\tau 
\end{align*}
and representations (\ref{cp6}) can be written as
\begin{align*}
y(t)=\frac{1}{\Gamma(\alpha)} \int_{0}^{t} (t-\tau)^{\alpha-1} F(t,\tau)f(\tau) d\tau+ F_{1}(t)y_{0}.
\end{align*}
\end{proof}
\bigskip
	\section{Statement of main result}
	In this section, we address an optimal control problem where the dynamical system is governed by nonlinear delayed equations featuring fractional Caputo derivatives. The system is subject to Cauchy conditions.
		\begin{equation}\label{T}
		\begin{cases}
			^c D^{\alpha}_{0+}y(t)=f(t,y(t),y(t-h),u(t)), \quad a.e. t\in [0,T]\\
			y(0)=y_{0}, \\
			y(t)=\varphi(t); \quad -h\leq t<0,\quad  h>0.
		\end{cases}
	\end{equation}
Consider the optimal control problem, where $y(t)$ represents an $n$-dimensional vector of phase variables, $u(t)$ is a measurable and bounded control function on the interval $[0,T]$ with $u(t) \in U(t) \subset R^r$, $y_0 \in R^n$ is a fixed initial condition, and $0 < \alpha < 1$. 

In this problem, the objective is to minimize the expected cost functional defined as follows.

	\begin{align}\label{F}
		J(u)=\Phi(y(T))+\frac{1}{\Gamma(\beta)}\int_{0}^{T}(T-\tau)^{\beta-1}f_{0}(t,y(t),y(t-h),u(t))dt\longrightarrow min.
	\end{align}
An admissible pair is referred to as the pair $(y(t),u(t))$ that satisfies the constraints (\ref{T}). On the other hand, an optimal pair denotes the pair $(y(t),u(t))$ that serves as a solution to the problem (\ref{T})-(\ref{F}).
\bigskip
\subsection*{Problem.}
Minimize (\ref{F}) over $U(t).$ Any $u(\cdot)\in U(t)$ satisfying 
$$
J(u(\cdot))=\inf_{v\in U(t)}J(v(\cdot))
$$
is called an optimal control.

$\bullet$(C1) Let $d:=f, f_{0}$. There exists a constant $L>0$ such that
$$
\Vert d(t,y_{1},y^{\prime}_{1},u_{1})- d(t,y_{2},y^{\prime}_{2},u_{2})\Vert\\
\leq L(\Vert y_{1}-y_{2}\vert +\Vert y^{\prime}_{1}-y^{\prime}_{2}\Vert +\Vert u_{1}-u_{2}\Vert),
$$
$$
\Vert d(t,0,0,0)\Vert \leq L,\quad t\in [0,T], \quad y_{1},y_{2},y^{\prime}_{1},y^{\prime}_{2}\in R^{n}, \quad u_{1},u_{2}\in R^{r};
$$

$\bullet$(C2) Let $d:=f, f_{0}$ be continuously differentiable with respect to $x,y,u$, and be measureable with respect to  $t$. Morover, there exists a constant $L_{1}>0$ such that
$$
\Vert d_{y}(t,y_{1},y^{\prime}_{1},u_{1})-d_{y}(t,y_{2},y^{\prime}_{2},u_{2})\Vert
$$
$$
\vee \Vert d_{y^{\prime}}(t,y_{1},y^{\prime}_{1},u_{1})-d_{y^{\prime}}(t,y_{2},y^{\prime}_{2},u_{2})\Vert
$$
$$
\vee\Vert d_{u}(t,y_{1},y^{\prime}_{1},u_{1})-d_{y}(t,y_{2},y^{\prime}_{2},u_{2})\Vert
$$
$$
\leq L_{1}(\Vert y_{1}-y_{2}\vert +\Vert y^{\prime}_{1}-y^{\prime}_{2}\Vert +\Vert u_{1}-u_{2}\Vert),
$$
$$
\Vert d_{y}(t,0,0,0)\Vert \vee \Vert d_{y}(t,0,0,0)\Vert \vee \Vert d_{u}(t,0,0,0)\Vert\leq L_{1}, \quad t\in [0,T], \quad  y_{1},y_{2},y^{\prime}_{1},y^{\prime}_{2}\in R^{n}, \quad u_{1},u_{2}\in R^{r};
$$

$\bullet$ (C3) $\Phi$ is continuously differentiable and uniformly bounded;

$\bullet$ (C4) The set ${U(t): t\in[0,T] }$ is compact.
\bigskip
\begin{definition}
	The conjugate problem of the $(4.1)-(4.2)$ is said the following integral equation: 
	\begin{align*}
		&\Psi(t)=-\frac{\partial{\Phi(y(T))}}{\partial{y(T)}}+	\frac{(T-t)^{1-\alpha}}{\Gamma(\alpha)}\int_{t}^{T}(T-\tau)^{\alpha-1}(\tau-t)^{\alpha-1}\frac{\partial{H(\tau,y(\tau),y(\tau-h),u(\tau),\Psi)}}{\partial{y(\tau)}}d\tau\\
		+&\frac{(T-t)^{1-\alpha}}{\Gamma(\alpha)}\int_{t+h}^{T}(T-\tau)^{\alpha-1}(\tau-t-h)^{\alpha-1}\frac{\partial{H(\tau,y(\tau),y(\tau-h),u(\tau),\Psi)}}{\partial{y(\tau-h)}}d\tau, \quad t\in [0,T],\\
		&\Psi(t)=0, \quad t\in (T,T+h].
	\end{align*}
\end{definition}
\begin{lemma}\label{lem1}
Let $(u(t),y(t))$ be optimal pair for the problem $(\ref{T})-(\ref{F}).$ Then, itcrement of trajectory and functional is as follows:
\begin{align*}
	&\Delta J(u)=-\frac{1}{\Gamma(\alpha)}\int_{0}^{T}(T-t)^{\alpha-1}\Delta_{\tilde{u}}H(t,y(t),y(t-h),u(t),\Psi)dt\nonumber\\
	-&\frac{1}{\Gamma(\alpha)}\int_{0}^{T}(T-t)^{\alpha-1}\frac{\partial{\Delta_{\tilde{u}}H(t,y(t),y(t-h),u(t),\Psi))}}{\partial(y(t))}\Delta y(t)dt\nonumber\\
	-&\frac{1}{\Gamma(\alpha)}\int_{0}^{T}(T-t)^{\alpha-1}\frac{\partial{\Delta_{\tilde{u}}H(t,y(t),y(t-h),u(t),\Psi))}}{\partial(y(t-h))}\Delta y(t-h)dt\nonumber\\
	-&\frac{1}{\Gamma(\alpha)}\int_{0}^{T}(T-t)^{\alpha-1}\bigg[\frac{1}{2}\Delta y^{\prime}(t)\frac{\partial^{2}{H(t,y(t)+\mu_{2}\Delta y(t),y(t-h)+\mu_{2}\Delta y(t-h),u(t),\Psi)}}{\partial y^{2}(t)}\Delta y(t)\nonumber\\
	+&\frac{1}{2}\Delta y^{\prime}(t-h)\frac{\partial^{2}{H(t,y(t)+\mu_{2}\Delta y(t),y(t-h)+\mu_{2}\Delta y(t-h),u(t),\Psi)}}{\partial y^{2}(t-h)}\Delta y(t-h)\\
	+&\frac{1}{2}\Delta y^{\prime}(t)\frac{\partial^{2}{H(t,y(t)+\mu_{2}\Delta y(t),y(t-h)+\mu_{2}\Delta y(t-h),u(t),\Psi)}}{\partial y(t)\partial y(t-h)}\Delta y(t-h)\bigg]dt\nonumber\\
	-&\frac{1}{\Gamma(\alpha)}\int_{0}^{T}(T-t)^{\alpha-1}\bigg[\frac{1}{2}\Delta y^{\prime}(t)\frac{\partial^{2}{\Delta_{\tilde{u}}H(t,y(t)+\mu_{2}\Delta y(t),y(t-h)+\mu_{2}\Delta y(t-h),u(t),\Psi)}}{\partial y^{2}(t)}\Delta y(t)\nonumber\\
	+&\frac{1}{2}\Delta y^{\prime}(t-h)\frac{\partial^{2}{\Delta_{\tilde{u}}H(t,y(t)+\mu_{2}\Delta y(t),y(t-h)+\mu_{2}\Delta y(t-h),u(t),\Psi)}}{\partial^{2} y(t-h)}\Delta y(t-h)\nonumber\\
	+&\frac{1}{2}\Delta y^{\prime}(t)\frac{\partial^{2}{\Delta_{\tilde{u}}H(t,y(t)+\mu_{2}\Delta y(t),y(t-h)+\mu_{2}\Delta y(t-h),u(t),\Psi)}}{\partial y(t)\partial y(t-h)}\Delta y(t-h)\bigg]dt\nonumber\\	
	+&\frac{1}{2}\Delta y^{\prime}(T)\frac{\partial^{2}{\Phi(y(T)+\mu_{1}\Delta y(T))}}{\partial{y^{2}(T)}}\Delta y(T).\nonumber	
\end{align*}
\end{lemma}
\begin{proof}
	To establish the necessary conditions for optimal pairs, we must introduce specific perturbations to the control variable and ascertain the resulting variations in both the state and the cost functional. This process is analogous to determining the Taylor expansion of the state and cost in relation to the control perturbation.$\\$
Let {$y(t),u(t)$} and {$\tilde{u}(t)+\Delta u(t),\tilde{y}(t)+\Delta y(t)$} be two admissible process. It is clear that the increment $\Delta y(t)$ satisfies the problem 
\begin{equation}\label{T1}
	\begin{cases}
		^c D^{\alpha}_{0+}\Delta y(t)=\Delta f(t,y(t),y(t-h),u(t)), \quad a.e. t\in [0,T], \\
		\Delta y(0)=0, \\
		\Delta y(t)=0;\quad -h\leq t<0,\quad  h>0.
	\end{cases}
\end{equation}

where $\quad$	$\Delta f(t,y(t),y(t-h),u(t))=f(t,\tilde{y}(t),\tilde{y}(t-h),\tilde{u}(t))- f(t,y(t),y(t-h),u(t))$. We can write increment of functional in the form 

\begin{align*}
	\Delta J(u)=J(\tilde{u})-J(u)=\Delta \Phi(y(T))+\frac{1}{\Gamma(\beta)}\int_{0}^{T}(T-t)^{\beta-1}\Delta f_{0}(t,y(t),y(t-h),u(t))dt.
\end{align*}
where $
\Delta \Phi(y)=\Phi(\tilde{y})-\Phi(y)$, and $ \quad  \Delta f_{0}(t,y(t),y(t-h),u(t))=f_{0}(t,\tilde{y}(t),\tilde{y}(t-h),\tilde{u}(t))- f_{0}(t,y(t),y(t-h),u(t)).\\$
We introduce a nontrivial vector-function $\Psi(t) \in R^{n}$. By utilizing the Taylor formula, we can express the increment of the functional as follows
\begin{align*}
	&\Delta J(u)=\Delta \Phi(y(T))+\frac{1}{\Gamma(\beta)}\int_{0}^{T}(T-t)^{\beta-1}\Delta f_{0}(t,y(t),y(t-h),u(t))\\
	+&\frac{1}{\Gamma(\alpha)}\int_{0}^{T}(T-t)^{\alpha-1}\Psi^{,}(t)\bigg[ (^c D^{\alpha}_{0+}\Delta y)(t)-\Delta f(t,y(t),y(t-h),u(t))\bigg]dt\\
	=&\Delta \Phi(y(T))-\frac{1}{\Gamma(\alpha)}\int_{0}^{T}(T-t)^{\alpha-1}\Delta H(t,y(t),y(t-h),u(t),\Psi)dt\\
	+&\frac{1}{\Gamma(\alpha)}\int_{0}^{T}(T-t)^{\alpha-1}\Psi(t)\bigg( (^c D^{\alpha}_{0+}\Delta y)(t)\bigg)dt\\
	=&\frac{\partial{\Phi(y(T))}}{\partial{y(T)}}\Delta y(T)+\frac{1}{2}\Delta y^{\prime}(T)\frac{\partial^{2}{\Phi(y(T)+\mu_{1}\Delta y(T))}}{\partial{y^{2}(T)}}\Delta y(T)+\frac{1}{\Gamma(\alpha)}\int_{0}^{T}(T-t)^{\alpha-1}\Psi(t)\bigg( (^c D^{\alpha}_{0+}\Delta y)(t)\bigg)dt\\
	-&\frac{1}{\Gamma(\alpha)}\int_{0}^{T}\Delta_{u} H(t,y(t),y(t-h),u(t),\Psi)dt\\
	-&\frac{1}{\Gamma(\alpha)}\int_{0}^{T}(T-t)^{\alpha-1}\bigg[\frac{\partial{H(t,y(t),y(t-h),u(t),\Psi)}}{\partial y(t)}\Delta y(t)+\frac{\partial{H(t,y(t),y(t-h),u(t),\Psi)}}{\partial y(t-h)}\Delta y(t-h)\bigg]dt\\
	-&\frac{1}{\Gamma(\alpha)}\int_{0}^{T}(T-t)^{\alpha-1}\bigg[\frac{\partial{\Delta_{\tilde{u}}H(t,y(t),y(t-h),u(t),\Psi)}}{\partial y(t)}\Delta y(t)+\frac{\partial{\Delta_{\tilde{u}}H(t,y(t),y(t-h),u(t),\Psi)}}{\partial y(t-h)}\Delta y(t-h)\bigg]dt\\
	-&\frac{1}{\Gamma(\alpha)}\int_{0}^{T}(T-t)^{\alpha-1}\bigg[\frac{1}{2}\Delta y^{\prime}(t)\frac{\partial^{2}{H(t,y(t)+\mu_{2}\Delta y(t),y(t-h)+\mu_{2}\Delta y(t-h),u(t),\Psi)}}{\partial y^{2}(t)}\Delta y(t)\\
	+&\frac{1}{2}\Delta y^{\prime}(t-h)\frac{\partial^{2}{H(t,y(t)+\mu_{2}\Delta y(t),y(t-h)+\mu_{2}\Delta y(t-h),u(t),\Psi)}}{\partial y^{2}(t-h)}\Delta y(t-h)\\
	+&\frac{1}{2}\Delta y^{\prime}(t)\frac{\partial^{2}{H(t,y(t)+\mu_{2}\Delta y(t),y(t-h)+\mu_{2}\Delta y(t-h),u(t),\Psi)}}{\partial y(t)\partial y(t-h)}\Delta y(t-h)\bigg]dt\\
	-&\frac{1}{\Gamma(\alpha)}\int_{0}^{T}(T-t)^{\alpha-1}\bigg[\frac{1}{2}\Delta y^{\prime}(t)\frac{\partial^{2}{\Delta_{\tilde{u}}H(t,y(t)+\mu_{2}\Delta y(t),y(t-h)+\mu_{2}\Delta y(t-h),u(t),\Psi)}}{\partial y^{2}(t)}\Delta y(t)\\
	+&\frac{1}{2}\Delta y^{\prime}(t-h)\frac{\partial^{2}{\Delta_{\tilde{u}}H(t,y(t)+\mu_{2}\Delta y(t),y(t-h)+\mu_{2}\Delta y(t-h),u(t),\Psi)}}{\partial^{2} y(t-h)}\Delta y(t-h)\\
	+&\frac{1}{2}\Delta y^{\prime}(t)\frac{\partial^{2}{\Delta_{\tilde{u}}H(t,y(t)+\mu_{2}\Delta y(t),y(t-h)+\mu_{2}\Delta y(t-h),u(t),\Psi)}}{\partial y(t)\partial y(t-h)}\Delta y(t-h)\bigg]dt.\\
\end{align*}
Note that 
$$
H(t,y,y_{h},u,\Psi))=\Psi^{'}f(t,y,y_{h},u)-\frac{\Gamma(\alpha)}{\Gamma(\beta)}(T-t)^{\beta-\alpha}f_{0}(t,y,y_{h},u).
$$
Using relations
\begin{align*}
	\Delta y(t)=\frac{1}{\Gamma(\alpha)} \int_{0}^{t} (t-\tau)^{\alpha-1}(^c D^{\alpha}_{0+}\Delta y)(\tau)d\tau, 
\end{align*}
and
\begin{align*}
	\Delta y(t-h)=\frac{1}{\Gamma(\alpha)} \int_{0}^{t-h} (t-h-\tau)^{\alpha-1}(^c D^{\alpha}_{0+}\Delta y)(\tau)d\tau
\end{align*}
we obtain that
\begin{align}\label{F1}
	&\Delta J(u)=\frac{1}{\Gamma(\alpha)}\int_{0}^{T}(T-t)^{\alpha-1}\bigg[\frac{\partial{\Phi(y(T))}}{\partial{y(T)}}+\Psi^{,}(t)\nonumber\\
	-&\frac{(T-t)^{1-\alpha}}{\Gamma(\alpha)}\int_{t}^{T}(T-t)^{\alpha-1} (t-\tau)^{\alpha-1}\frac{\partial{H(\tau,y(\tau),y(\tau-h),u(\tau),\Psi)}}{\partial{y(\tau)}}d\tau\nonumber\\
	-&\frac{(T-t)^{1-\alpha}}{\Gamma(\alpha)}\int_{t+h}^{T}(T-t)^{\alpha-1}(t-\tau-h)^{\alpha-1}\frac{\partial{H(\tau,y(\tau),y(\tau-h),u(\tau),\Psi)}}{\partial{y(\tau-h)}}d\tau\bigg](^c D^{\alpha}_{0+}\Delta y)(t)dt\nonumber\\
	-&\frac{1}{\Gamma(\alpha)}\int_{0}^{T}(T-t)^{\alpha-1}\Delta_{\tilde{u}}H(t,y(t),y(t-h),u(t),\Psi)dt\nonumber\\
	-&\frac{1}{\Gamma(\alpha)}\int_{0}^{T}(T-t)^{\alpha-1}\bigg[\frac{\partial{\Delta_{\tilde{u}}H(t,y(t),y(t-h),u(t),\Psi)}}{\partial y(t)}\Delta y(t)+\frac{\partial{\Delta_{\tilde{u}}H(t,y(t),y(t-h),u(t),\Psi)}}{\partial y(t-h)}\Delta y(t-h)\bigg]dt\nonumber\\
	-&\frac{1}{\Gamma(\alpha)}\int_{0}^{T}(T-t)^{\alpha-1}\bigg[\frac{1}{2}\Delta y^{\prime}(t)\frac{\partial^{2}{H(t,y(t)+\mu_{2}\Delta y(t),y(t-h)+\mu_{2}\Delta y(t-h),u(t),\Psi)}}{\partial y^{2}(t)}\Delta y(t)\nonumber\\
	+&\frac{1}{2}\Delta y^{\prime}(t-h)\frac{\partial^{2}{H(t,y(t)+\mu_{2}\Delta y(t),y(t-h)+\mu_{2}\Delta y(t-h),u(t),\Psi)}}{\partial y^{2}(t-h)}\Delta y(t-h)\\
	+&\frac{1}{2}\Delta y^{\prime}(t)\frac{\partial^{2}{H(t,y(t)+\mu_{2}\Delta y(t),y(t-h)+\mu_{2}\Delta x(t-h),u(t),\Psi)}}{\partial y(t)\partial y(t-h)}\Delta y(t-h)\bigg]dt\nonumber\\
	-&\frac{1}{\Gamma(\alpha)}\int_{0}^{T}(T-t)^{\alpha-1}\bigg[\frac{1}{2}\Delta y^{\prime}(t)\frac{\partial^{2}{\Delta_{\tilde{u}}H(t,y(t)+\mu_{2}\Delta y(t),y(t-h)+\mu_{2}\Delta y(t-h),u(t),\Psi)}}{\partial y^{2}(t)}\Delta y(t)\nonumber\\
	+&\frac{1}{2}\Delta y^{\prime}(t-h)\frac{\partial^{2}{\Delta_{\tilde{u}}H(t,y(t)+\mu_{2}\Delta y(t),y(t-h)+\mu_{2}\Delta y(t-h),u(t),\Psi)}}{\partial^{2} y(t-h)}\Delta y(t-h)\nonumber\\
	+&\frac{1}{2}\Delta y^{\prime}(t)\frac{\partial^{2}{\Delta_{\tilde{u}}H(t,y(t)+\mu_{2}\Delta y(t),y(t-h)+\mu_{2}\Delta y(t-h),u(t),\Psi)}}{\partial y(t)\partial y(t-h)}\Delta y(t-h)\bigg]dt\nonumber\\
	+&\frac{1}{2}\Delta y^{\prime}(T)\frac{\partial^{2}{\Phi(y(T)+\mu_{1}\Delta y(T))}}{\partial{y^{2}(T)}}\Delta y(T), \quad where\quad 0<\mu_{i}<1, \quad i=1,2.\nonumber
\end{align}
Next, we impose the condition that the vector-function $\Psi:[0,T]\to R^{n},$ satisfies the following integral equation
\begin{align}\label{con}
	&\Psi(t)=-\frac{\partial{\Phi(y(T))}}{\partial{y(T)}}+	\frac{(T-t)^{1-\alpha}}{\Gamma(\alpha)}\int_{t}^{T}(T-\tau)^{\alpha-1}(\tau-t)^{\alpha-1}\frac{\partial{H(\tau,y(\tau),y(\tau-h),u(\tau),\Psi)}}{\partial{y(\tau)}}d\tau\\
	+&\frac{(T-t)^{1-\alpha}}{\Gamma(\alpha)}\int_{t+h}^{T}(T-\tau)^{\alpha-1}(\tau-t-h)^{\alpha-1}\frac{\partial{H(\tau,y(\tau),y(\tau-h),u(\tau),\Psi)}}{\partial{y(\tau-h)}}d\tau, \quad t\in [0,T],\nonumber\\
	&\Psi(t)=0, \quad t\in (T,T+h].\nonumber
\end{align}
The problem stated in equation (\ref{con}) is referred to as a conjugated problem. The integral equation (\ref{con}) possesses a single continuous solution. By considering the equalities expressed in (\ref{con}) within (\ref{F1}), we obtain an auxiliary formula for determining the increment of the functional.
\begin{align}\label{F2}
	&\Delta J(u)=-\frac{1}{\Gamma(\alpha)}\int_{0}^{T}(T-t)^{\alpha-1}\Delta_{\tilde{u}}H(t,y(t),y(t-h),u(t),\Psi)dt\nonumber\\
	-&\frac{1}{\Gamma(\alpha)}\int_{0}^{T}(T-t)^{\alpha-1}\frac{\partial{\Delta_{\tilde{u}}H(t,y(t),y(t-h),u(t),\Psi)}}{\partial(y(t))}\Delta y(t)dt\nonumber\\
	-&\frac{1}{\Gamma(\alpha)}\int_{0}^{T}(T-t)^{\alpha-1}\frac{\partial{\Delta_{\tilde{u}}H(t,y(t),y(t-h),u(t),\Psi)}}{\partial(y(t-h))}\Delta y(t-h)dt\nonumber\\
	-&\frac{1}{\Gamma(\alpha)}\int_{0}^{T}(T-t)^{\alpha-1}\bigg[\frac{1}{2}\Delta y^{\prime}(t)\frac{\partial^{2}{H(t,y(t)+\mu_{2}\Delta y(t),y(t-h)+\mu_{2}\Delta y(t-h),u(t),\Psi)}}{\partial y^{2}(t)}\Delta y(t)\nonumber\\
	+&\frac{1}{2}\Delta y^{\prime}(t-h)\frac{\partial^{2}{H(t,y(t)+\mu_{2}\Delta y(t),y(t-h)+\mu_{2}\Delta y(t-h),u(t),\Psi)}}{\partial y^{2}(t-h)}\Delta y(t-h)\\
	+&\frac{1}{2}\Delta y^{\prime}(t)\frac{\partial^{2}{H(t,y(t)+\mu_{2}\Delta y(t),y(t-h)+\mu_{2}\Delta y(t-h),u(t),\Psi)}}{\partial y(t)\partial y(t-h)}\Delta y(t-h)\bigg]dt\nonumber\\
	-&\frac{1}{\Gamma(\alpha)}\int_{0}^{T}(T-t)^{\alpha-1}\bigg[\frac{1}{2}\Delta y^{\prime}(t)\frac{\partial^{2}{\Delta_{\tilde{u}}H(t,y(t)+\mu_{2}\Delta y(t),y(t-h)+\mu_{2}\Delta y(t-h),u(t),\Psi)}}{\partial y^{2}(t)}\Delta y(t)\nonumber\\
	+&\frac{1}{2}\Delta y^{\prime}(t-h)\frac{\partial^{2}{\Delta_{\tilde{u}}H(t,y(t)+\mu_{2}\Delta y(t),y(t-h)+\mu_{2}\Delta y(t-h),u(t),\Psi)}}{\partial^{2} y(t-h)}\Delta y(t-h)\nonumber\\
	+&\frac{1}{2}\Delta y^{\prime}(t)\frac{\partial^{2}{\Delta_{\tilde{u}}H(t,y(t)+\mu_{2}\Delta y(t),y(t-h)+\mu_{2}\Delta y(t-h),u(t),\Psi)}}{\partial y(t)\partial y(t-h)}\Delta y(t-h)\bigg]dt\nonumber\\	
	+&\frac{1}{2}\Delta y^{\prime}(T)\frac{\partial^{2}{\Phi(y(T)+\mu_{1}\Delta y(T))}}{\partial{y^{2}(T)}}\Delta y(T).\nonumber	
\end{align}
\end{proof}
\bigskip
\begin{lemma}\label{lem}
	Under the assumptions of Lemma \ref{lem1} and $(C1)-(C4),$ the following estimate holds for trajectory increment on the spike variation:
	\begin{align*}
		\Vert\Delta y(t)\Vert\leq B\varepsilon,
	\end{align*}
	where   $\quad B=M\bigg(\frac{K_{1}(\alpha,\varepsilon)}{\Gamma(\alpha+1)}(t-(\theta+\varepsilon))^{\alpha-1}+\frac{MK(v)}{\Gamma(\alpha+1)}\varepsilon^{\alpha-1}\bigg).$
\end{lemma}
\begin{proof}
	In order to obtain effectively verifiable necessary optimal conditions, consider the following variation of $u(t):$
	\begin{align}\label{I}
		\Delta_{v}u(t)=
		\begin{cases}
			0,       & \quad t \notin [\theta,\theta+\varepsilon) \\
			v-u(t),          & \quad t\in  [\theta,\theta+\varepsilon).
		\end{cases}
	\end{align}
	where $v \in U, \varepsilon>0$ is a sufficiently small parameter, $\theta \in [0,T),\quad \theta+\varepsilon<T,$ is the Lebesgue point of the functions $t\to f(t,y(t),y(t-h),u(t)),f_{y}(t,y(t),y(t-h),u(t)).$
	For all values of $\varepsilon>0$ that are sufficiently small, the control $u_{\varepsilon}(t)$ defined as $u_{\varepsilon}(t)=u(t)+\Delta_{\varepsilon}u(t)$ is considered to be an admissible control. This control, obtained by perturbing the original control $u(t)$, is valid for all $v\in U.$
	\begin{equation}\label{T2}
		\begin{cases}
			^c D^{\alpha}_{0+}\Delta y(t)=f(t,y(t)+\Delta y(t),y(t-h)+\Delta y(t-h),u(t))-f(t,y(t),y(t-h),u(t))\, \\
			\Delta y(t)=0, \quad -h\leq t\leq \theta, \theta+\varepsilon\leq T.\\
		\end{cases}
	\end{equation}
	As a consequence of $\Delta y(t)=0$, it can be deduced from equation (\ref{T2}) that the function
	$$\Delta y(t)\equiv0 \quad 0\leq t\leq \theta$$
	serves as the only solution to equation (\ref{T2}) within the interval $[0,\theta]$. 
	\begin{align}\label{T3}
		^c D^{\alpha}_{0+}\Delta y(t)=f(t,y(t)+\Delta y(t),y(t-h)+\Delta y(t-h),v)-f(t,y(t),y(t-h),u(t)),
		\quad \theta \leq t\leq \theta+\varepsilon.
	\end{align}
	Hence, the initial condition for equation (\ref{T3}) takes the form $\Delta y(\theta)=0$. Consequently, equation (\ref{T3}) can be expressed in integral form.    
	\begin{align}\label{s}
		\Delta y(t)=\frac{1}{\Gamma(\alpha)}\int_{\theta}^{t}(t-\tau)^{\alpha-1}[f(\tau,y(\tau)+\Delta y(\tau),y(\tau-h)+\Delta y(\tau-h),v)-f(\tau,y(\tau),y(\tau-h),u(\tau))]d\tau. 
	\end{align}
	According to the assumption $C1$, there exists a constant $L>0$ such that 
	\begin{equation}\label{L} 
		\Vert f(t,y(t)+\Delta y(t),y(t-h)+\Delta y(t-h),v)-f(t,y(t),y(t-h),u(t))\Vert\leq L[\Vert\Delta y(t)\Vert+	\Vert\Delta y(t-h)\Vert].
	\end{equation}
	from equation (\ref{s}) we obtain
	\begin{align}\label{s1}
		\Vert\Delta y(t)\Vert\leq\frac{1}{\Gamma(\alpha)}\int_{\theta}^{t}(t-\tau)^{\alpha-1}\Vert f(\tau,y(\tau)+\Delta y(\tau),y(\tau-h)+\Delta y(\tau-h),v)-f(\tau,y(\tau),y(\tau-h),u(\tau))\Vert d\tau\nonumber\\
		\leq \frac{L}{\Gamma(\alpha)}\int_{\theta}^{t}(t-\tau)^{\alpha-1}[\Vert\Delta y(\tau)\Vert+\Vert\Delta y(\tau-h)\Vert_{C}]d\tau+\frac{L}{\Gamma(\alpha)}\int_{\theta}^{t}(t-\tau)^{\alpha-1}\Vert\Delta_{v}f(\tau,y(\tau),y(\tau-h),u(\tau)) \Vert d\tau
	\end{align}
	where $\quad \Delta_{v}f(\tau,y(\tau),y(\tau-h),u(\tau)) =f(\tau,y(\tau),y(\tau-h),v)-f(\tau,y(\tau),y(\tau-h),u).\\$
	Denoting $$\esssup_{t \in [0,T]}\Vert\Delta_{v}f(\tau,y(\tau),y(\tau-h),u(\tau))\Vert=K(v),$$ the inequality (4.12) can be expressed in the following form. 
	\begin{align}\label{s2}
		\Vert\Delta y(t)\Vert\leq\frac{K(v)}{\Gamma(\alpha+1)}(t-\theta)^{\alpha}+\frac{L}{\Gamma(\alpha)}\int_{\theta}^{t}(t-\tau)^{\alpha-1}[\Vert\Delta y(\tau)\Vert+\Vert\Delta y(\tau-h)\Vert_{C}]d\tau.
	\end{align}
	\begin{align*}
		\Delta y(\theta)=0
	\end{align*}
	By applying Lemma \ref{Lmm} (Gronwall's inequality) to equation (\ref{s2}), we can conclude that for all $0 < \varepsilon \leq \tilde{\varepsilon}$ and $\theta \leq t \leq \theta + \varepsilon$, the following inequality holds.
	\begin{align}\label{s3}
		\Vert\Delta y(t)\Vert\leq\frac{MK(v)}{\Gamma(\alpha+1)}(t-\theta)^{\alpha} 
	\end{align}
	holds.$\\$
	Let us examine the behavior of the function $\Delta y(t)$ on the interval $[\theta+\varepsilon,T]$. This can be described by the integral equation
	\begin{align*}
		\Delta y(t)=\frac{1}{\Gamma(\alpha)}\int_{\theta}^{\theta+\varepsilon}(t-s)^{\alpha-1}[f(s,y(s)+\Delta y(s),y(s-h)+\Delta y(s-h),v)-f(s,y(s),y(s-h),u(s))]ds\\
		+\frac{1}{\Gamma(\alpha)}\int_{\theta+\varepsilon}^{t}(t-s)^{\alpha-1}[f(s,y(s)+\Delta y(s),y(s-h)+\Delta y(s-h),u(s))-f(s,y(s),y(s-h),u(s))]ds.
	\end{align*}
	To simplify the analysis, assuming the Lipschitz constant remains constant, we obtain the following result
	\begin{align*}
		&\Vert\Delta y(t)\Vert\leq\frac{K(v)}{\Gamma(\alpha)}\int_{\theta}^{\theta+\varepsilon}(t-s)^{\alpha-1}ds+\frac{L}{\Gamma(\alpha)}\int_{\theta}^{\theta+\varepsilon}(t-s)^{\alpha-1}\Vert\Delta y(s)\Vert ds\\
		+&\frac{L}{\Gamma(\alpha)}\int_{\theta}^{\theta+\varepsilon}(t-s)^{\alpha-1}\Vert\Delta y(s-h)\Vert_{C} ds+\frac{L}{\Gamma(\alpha)}\int_{\theta+\varepsilon}^{t}(t-s)^{\alpha-1}\Vert\Delta y(s)\Vert ds\\
		+&\frac{L}{\Gamma(\alpha)}\int_{\theta+\varepsilon}^{t}(t-s)^{\alpha-1}\Vert\Delta y(s-h)\Vert_{C} ds\leq\frac{K(v)}{\Gamma(\alpha)}\int_{\theta}^{\theta+\varepsilon}(t-s)^{\alpha-1}ds\\
		+&\frac{L}{\Gamma(\alpha)}\frac{MK(v)}{\Gamma(\alpha+1)}\int_{\theta}^{\theta+\varepsilon}(t-s)^{\alpha-1}(s-\theta)^{\alpha} ds	+\frac{L}{\Gamma(\alpha)}\frac{MK(v)}{\Gamma(\alpha+1)}\int_{\theta}^{\theta+\varepsilon}(t-s)^{\alpha-1}(s-\theta-h)^{\alpha} ds\\
		+&\frac{L}{\Gamma(\alpha)}\int_{\theta+\varepsilon}^{t}(t-s)^{\alpha-1}\Vert\Delta y(s)\Vert ds+\frac{L}{\Gamma(\alpha)}\int_{\theta+\varepsilon}^{t}(t-s)^{\alpha-1}\Vert\Delta y(s-h)\Vert_{C} ds
	\end{align*}
	Substituting (\ref{s3}) into the above expression, we obtain
	\begin{align*}
		&\Vert\Delta y(t)\Vert\leq\frac{K_{1}(\alpha,\varepsilon)\varepsilon}{\Gamma(\alpha+1)}(t-(\theta+\varepsilon))^{\alpha-1}+\frac{L}{\Gamma(\alpha)}\int_{\theta+\varepsilon}^{t}(t-s)^{\alpha-1}\Vert\Delta y(s)\Vert ds\\
		+&\frac{L}{\Gamma(\alpha)}\int_{\theta+\varepsilon}^{t}(t-s)^{\alpha-1}\Vert\Delta y(s-h)\Vert_{C} ds
	\end{align*}
	\begin{align*}
		\Vert\Delta y(\theta+\varepsilon)\Vert\leq\frac{MK(v)}{\Gamma(\alpha+1)}\varepsilon^{\alpha}
	\end{align*}
	Hence , by Lemma \ref{Lmm}, we have
	\begin{align*}
		\Vert\Delta y(t)\Vert\leq M\bigg(\frac{K_{1}(\alpha,\varepsilon)\varepsilon}{\Gamma(\alpha+1)}(t-(\theta+\varepsilon))^{\alpha-1}+\frac{MK(v)}{\Gamma(\alpha+1)}\varepsilon^{\alpha}\bigg)=B\varepsilon
	\end{align*}
\end{proof}
\begin{lemma}
	Under assumptions of the Lemma \ref{lem}, increment of functional on the spike variation is as follows:
	\begin{align*}
		&\Delta J(u)=-\frac{1}{\Gamma(\alpha)}\int_{\theta}^{\theta+\varepsilon}(T-t)^{\alpha-1}\Delta_{v}H(t)dt\\	-&\frac{\varepsilon^{\alpha+1}}{\Gamma(\alpha)\Gamma(\alpha+1)}\bigg[(T-\theta)^{\alpha-1}\Delta_{v}H_{y}(\theta)\Delta_{v}f(\theta)+(T-\theta-h)^{\alpha-1}\Delta_{v}H_{y_{h}}(\theta)\Delta_{v}f(\theta)\bigg]+o(\varepsilon^{1+\alpha}).
	\end{align*} 
\end{lemma}
\begin{proof}
	By utilizing the integral representation
	\begin{align*}
		\Delta y(t)=\frac{1}{\Gamma(\alpha)} \int_{\theta}^{t} (t-\tau)^{\alpha-1}[f(\tau,y(\tau)+\Delta y(\tau),,y(\tau-h)+\Delta y(\tau-h),v)-f(\tau,y(\tau),y(\tau-h),u(\tau))]d\tau,
	\end{align*}
	and (lemma \ref{Lmm2}), the increment of the trajectory on the interval $[\theta,\theta+\varepsilon],$ can be 
	\begin{align}\label{s4}
			\Delta y(t)=\frac{1}{\Gamma(\alpha+1)}(t-\theta)^{\alpha}\Delta_{v}f(\theta)+o((t-\theta)^{\alpha}).
	\end{align}
	By considering equations (\ref{s4}) and (\ref{I}) and incorporating them into equation (\ref{F2}), we derive the following result.
	\begin{align}\label{F3}
		&\Delta J(u)=-\frac{1}{\Gamma(\alpha)}\int_{\theta}^{\theta+\varepsilon}(T-t)^{\alpha-1}\Delta_{v}H(t,y(t),y(t-h),u(t),\Psi)dt\nonumber\\
		-&\frac{1}{\Gamma(\alpha)\Gamma(\alpha+1)}\int_{\theta}^{\theta+\varepsilon}(T-t)^{\alpha-1}(t-\theta)^{\alpha}\frac{\partial{\Delta_{v}H(t,y(t),y(t-h),u(t),\Psi)}}{\partial y(t)}\Delta_{v}f(\theta)dt\\
		-&\frac{1}{\Gamma(\alpha)\Gamma(\alpha+1)}\int_{\theta}^{\theta+\varepsilon}(T-t-h)^{\alpha-1}(t-\theta)^{\alpha}\frac{\partial{\Delta_{v}H(t+h,y(t+h),y(t),u(t+h),\Psi(t+h))}}{\partial y(t)}\Delta_{v}f(\theta)dt+\xi\nonumber
	\end{align}
	where
	\begin{align*}
		\xi=	-&\frac{1}{2\Gamma(\alpha)}\int_{\theta}^{\theta+\varepsilon}(T-t)^{\alpha-1}\bigg[\Delta y^{\prime}(t)\frac{\partial^{2}{H(t,y(t)+\mu_{2}\Delta y(t),y(t-h)+\mu_{2}\Delta y(t-h),u(t),\Psi)}}{\partial y^{2}(t)}\Delta y(t)\\
		+&\Delta y^{\prime}(t-h)\frac{\partial^{2}{H(t,y(t)+\mu_{2}\Delta y(t),y(t-h)+\mu_{2}\Delta y(t-h),u(t),\Psi)}}{\partial y^{2}(t-h)}\Delta y(t-h)\\
		+&\Delta y^{\prime}(t)\frac{\partial^{2}{H(t,y(t)+\mu_{2}\Delta y(t),y(t-h)+\mu_{2}\Delta y(t-h),u(t),\Psi)}}{\partial y(t)\partial y(t-h)}\Delta y(t-h)\bigg]dt\\
		-&\frac{1}{2\Gamma(\alpha)}\int_{\theta+\varepsilon}^{T}(T-t)^{\alpha-1}\bigg[\Delta y^{\prime}(t)\frac{\partial^{2}{H(t,y(t)+\mu_{2}\Delta y(t),y(t-h)+\mu_{2}\Delta y(t-h),u(t),\Psi)}}{\partial y^{2}(t)}\Delta y(t)\\
		+&\Delta y^{\prime}(t-h)\frac{\partial^{2}{H(t,y(t)+\mu_{2}\Delta y(t),y(t-h)+\mu_{2}\Delta y(t-h),u(t),\Psi)}}{\partial y^{2}(t-h)}\Delta y(t-h)\\
		+&\Delta y^{\prime}(t)\frac{\partial^{2}{H(t,y(t)+\mu_{2}\Delta y(t),y(t-h)+\mu_{2}\Delta y(t-h),u(t),\Psi)}}{\partial y(t)\partial y(t-h)}\Delta y(t-h)\bigg]dt\\
		-&\frac{1}{\Gamma(\alpha)}\int_{\theta}^{\theta+\varepsilon}(T-t)^{\alpha-1}\bigg[\frac{1}{2}\Delta y^{\prime}(t)\frac{\partial^{2}{\Delta_{\tilde{u}}H(t,y(t)+\mu_{2}\Delta y(t),y(t-h)+\mu_{2}\Delta y(t-h),u(t),\Psi)}}{\partial y^{2}(t)}\Delta y(t)\\
		+&\Delta y^{\prime}(t-h)\frac{\partial^{2}{\Delta_{\tilde{u}}H(t,y(t)+\mu_{2}\Delta y(t),y(t-h)+\mu_{2}\Delta y(t-h),u(t),\Psi)}}{\partial^{2} y(t-h)}\Delta y(t-h)\\
		+&\Delta y^{\prime}(t)\frac{\partial^{2}{\Delta_{\tilde{u}}H(t,y(t)+\mu_{2}\Delta y(t),y(t-h)+\mu_{2}\Delta y(t-h),u(t),\Psi)}}{\partial y(t)\partial y(t-h)}\Delta y(t-h)\bigg]dt\\
		+&\frac{1}{2}\Delta y^{\prime}(T)\frac{\partial^{2}{\Phi(y(T)+\mu_{1}\Delta y(T))}}{\partial{y^{2}(T)}}\Delta y(T).
	\end{align*}
	by the same way in \cite{37}, it follows 
	\begin{align}\label{F4}
		\vert \xi\vert\leq K\varepsilon^{1+\nu}, \quad \nu>0.
	\end{align}
	it implies
	\begin{align}\label{F5}
		-&\frac{1}{\Gamma(\alpha)\Gamma(\alpha+1)}\int_{\theta}^{\theta+\varepsilon}(T-t)^{\alpha-1}(t-\theta)^{\alpha}\frac{\partial{\Delta_{v}H(t,y(t),y(t-h),u(t),\Psi)}}{\partial y(t)}\Delta_{v}f(\theta)dt\nonumber\\
		-&\frac{1}{\Gamma(\alpha)\Gamma(\alpha+1)}\int_{\theta}^{\theta+\varepsilon}(T-t-h)^{\alpha-1}(t-\theta)^{\alpha}\frac{\partial{\Delta_{v}H(t,y(t),y(t-h),u(t),\Psi)}}{\partial y(t-h)}\Delta_{v}f(\theta)dt\\
		=-&\frac{\varepsilon^{\alpha+1}}{\Gamma(\alpha)\Gamma(\alpha+1)}\bigg[(T-\theta)^{\alpha-1}\Delta_{v}H_{y}(\theta)\Delta_{v}f(\theta)+(T-\theta-h)^{\alpha-1}\Delta_{v}H_{y_{h}}(\theta+h)\Delta_{v}f(\theta)\bigg]+o(\varepsilon^{1+\alpha}).\nonumber
	\end{align}
	After substituting the estimates (\ref{F4}) and (\ref{F5}) into the expression (\ref{F3}) representing the increment of the functional, we obtain the following result
	\begin{align}\label{F6}
		&\Delta J(u)=-\frac{1}{\Gamma(\alpha)}\int_{\theta}^{\theta+\varepsilon}(T-t)^{\alpha-1}\Delta_{v}H(t)dt\nonumber\\	-&\frac{\varepsilon^{\alpha+1}}{\Gamma(\alpha)\Gamma(\alpha+1)}\bigg[(T-\theta)^{\alpha-1}\Delta_{v}H_{y}(\theta)\Delta_{v}f(\theta)+(T-\theta-h)^{\alpha-1}\Delta_{v}H_{y_{h}}(\theta+h)\Delta_{v}f(\theta)\bigg]+o(\varepsilon^{1+\alpha}).
	\end{align}
\end{proof} 
\subsection*{First ans second order necessary conditions.}
\bigskip
\begin{theorem}\label{ll} 
Suppose that the admissible process $(u(t),y(t))$ is optimal for the problem (\ref{T})-(\ref{F}), and let $\Psi$ be a solution to the conjugate problem (\ref{con}) on the optimal process. Assuming that assumptions (C1)-(C4) are satisfied, then for almost every $t \in [0,T]$, the following equality is satisfied:
\begin{align}\label{H}
\max_{v\in U}H(t,y(t),y(t-h),v,\Psi(t))=H(t,y(t),y(t-h),u(t),\Psi(t)).
\end{align}
\end{theorem}
\bigskip
\begin{proof}
When we divide the right side of (\ref{F6}) by $\varepsilon$ and let $\varepsilon$ approach zero from the positive side, we obtain (\ref{H}).
\end{proof}
\bigskip
\begin{definition}
An admissible control $u(t)$ is considered to be singular according to the Pontryagin maximum principle if, within the process ${u(t),y(t)}$, the maximum condition (\ref{T1}) is trivially satisfied on a subset $T_0 \subset [0,T]$. In other words, it implies that
\begin{align}\label{H1}
	\Delta_v H(t) = 0 \quad \text{for almost every} \quad t \in T_0 \quad \text{and} \quad v \in U_0.
\end{align}
where $U_{0}\not=\lbrace u(t)\rbrace,$ $t\in T_{0}$ is a subset of $U.$
\end{definition}
For simplicity of presentation, we assume that $U_{0}=U$ and $T_{0}=[0,T].$
\begin{theorem}\label{ln}
	In order for a control $u(t)$ to be optimal according to the Pontryagin maximum principle, in the context of the problem (\ref{T})-(\ref{F}), it is necessary that the following inequality holds:
	\begin{align}\label{H2}
	(T-t)^{\alpha-1}\Delta_{v}H_{y}(t)\Delta_{v}f(t)+(T-t-h)^{\alpha-1}\Delta_{v}H_{y_{h}}(t+h)\Delta_{v}f(t)\leq0,
	\end{align}
	holds for all $v\in U$ and a.e.\quad $t\in [0,T).$
\end{theorem}
\bigskip
\begin{proof} 
Now, considering  $\Delta_v H(t) = 0 $ in expression (\ref{F6}), we get
	\begin{align}\label{F7}
	&\Delta J(u)=\nonumber\\
	-&\frac{\varepsilon^{\alpha+1}}{\Gamma(\alpha)\Gamma(\alpha+1)}\bigg[(T-\theta)^{\alpha-1}\Delta_{v}H_{y}(\theta)\Delta_{v}f(\theta)+(T-\theta-h)^{\alpha-1}\Delta_{v}H_{y_{h}}(\theta+h)\Delta_{v}f(\theta)\bigg]+o(\varepsilon^{1+\alpha}).
\end{align}
dividing the right side of (\ref{F7}) by $\varepsilon^{\alpha+1}$ and let  $\varepsilon$ approach zero from the positive side, we get (\ref{H2}).
\end{proof}
\bigskip
\begin{example}
	Consider the problem
	\begin{align*}
		&(^c D^{\alpha}_{0+}y)(t)=Ay(t-h)+Bu(t),\quad t\in [0,1],\\
		&y(t)=0,  \quad -h\leq t\leq0, \quad h=\frac{1}{2},\\
		&\vert u(t)\vert\leq 1, \quad  J(u)=y_{1}(1)\longrightarrow min,\\
	\end{align*}
where 
\begin{align*}
	A=\begin{bmatrix}
		0&1\\  0&0
	\end{bmatrix}
\quad  B=\begin{bmatrix}
	1\\  1
\end{bmatrix}.
\end{align*}
We analyze the optimality of the control input $u(t)=0$. This particular control choice corresponds to the solution $y(t)=0$ of the equation and Hamiltonian of the problem is as follows
\begin{align*}
	&H(t,y(t),y(t-h),u(t),\psi)=\langle \psi,f(t,y(t),y(t-h),u(t))\rangle =\psi_{1}(y_{1}(t-h)+u(t))+\psi_{2} u(t),\\
	&where \quad \psi=(\psi{_1},\psi_{2}) \quad and \quad y(t-h)=(y_{1}(t-h),y_{2}(t-h)).
\end{align*}
 During the progression denoted by $(0.0)$, we obtain the following outcome
\begin{align*}
	\psi_{1}(t)=-1, \quad \psi_{2}(t)=-\frac{\Gamma (\alpha)}{\Gamma (2\alpha)}(1-t)^{1-\alpha}(1-t-h)^{2\alpha-1},\quad t\in\bigg[0,\frac{1}{2}\bigg]
\end{align*}
According to the theorem \ref{ll}, we have
\begin{align*}
	-\bigg(1+\frac{\Gamma (\alpha)}{\Gamma (2\alpha)}(1-t)^{1-\alpha}(1-t-h)^{2\alpha-1}\bigg)v\leq0, \quad v\in [-1,1],
\end{align*}
that is not possible for all $v\in [-1,0).$ Therefore, $u(t)=0$ is not optimal control.
\end{example}
\bigskip
\begin{example}
	Consider the problem
\begin{align*}
	&(^c D^{\alpha}_{0+}y)(t)=[A+By_{1}(t-h)]u(t),\quad t\in [0,1],\\
	&y(t)=0,\quad t\in[-h,0]  \quad h=\frac{1}{2},\\
	&\vert u(t)\vert\leq 1, \quad  J(u)=y_{2}(1)\longrightarrow min,\\
\end{align*}
\end{example}
where 
\begin{align*}
	A=\begin{bmatrix}
		1\\  0
	\end{bmatrix}
	\quad  B=\begin{bmatrix}
		0\\  -1
	\end{bmatrix}.
\end{align*}
We are examining the effectiveness of the control input $u(t) = 0$ and assessing its optimality. This specific choice of control corresponds to the solution $y(t) = 0$ of the equation. The Hamiltonian of the problem can be expressed as follows.
\begin{align*}
	&H(t,y(t),y(t-h),u(t),\psi)=\langle \psi,f(t,y(t),y(t-h),u(t))\rangle =\psi_{1}u(t)-\psi_{2} u(t)y_{1}(t-h),\\
	&where \quad \psi=(\psi{_1},\psi_{2}) \quad and \quad y(t-h)=(y_{1}(t-h),y_{2}(t-h)).
\end{align*}
 Along the process $(0.0)$, we observe the following result.
\begin{align*}
	\psi_{1}(t)=0, \quad \psi_{2}(t)=-1,\quad and \quad H=0.
\end{align*}
Hence, the control $u(t) = 0$ is categorized as a singular control. We will now verify the satisfaction of the condition in equation (\ref{H2}). By examining equation (\ref{H2}), we can deduce that $v^2 \leq 0$ for all $v \in [-1, 1]$, which is not feasible for any $v \neq 0$. Consequently, the control $u(t) = 0$ is not considered optimal.

Clearly, when applying the control $u(t) = 0$, the value of the performance measure is $J(t) = y_2(1) = 0$. Now, let's examine whether there exists another control function that results in objective functional values less than zero. We will calculate the value of $J$ for the admissible control $u(t) = -\frac{1}{2}.$ For this function
\begin{align*}
	y_{1}(t)=-\frac{t^{\alpha}}{2\Gamma(\alpha+1)},
\end{align*}
\begin{align*}
	y_{2}(t)=-\frac{(t-h)^{2\alpha}}{4\Gamma(2\alpha+1)}.
\end{align*}
Then we have 
\begin{align*}
	J\bigg(-\frac{1}{2}\bigg)=-\frac{1}{2^{2\alpha+2}\Gamma(2\alpha+1)}<0=J(0).
\end{align*}
This demonstrates that the choice of control $u(t)=0$ for $ t\in[0,1],$ is not optimal.
\bigskip

\section{Conclusion}
By employing an updated iteration technique, we successfully derive a first-order condition for optimality,
represented through the Pontryagin maximum principle, for the given fractional time-delay optimal control
problem. Furthermore, in situations where the Pontryagin maximum principle degenerates, we also obtain a
second-order necessary condition for optimality. As a result, we anticipate that this approach can be extended
and generalized to address fractional optimal control problems involving neutral fractional equations, partial
fractional differential equations, along with their respective boundary conditions, in future research.
           
\end{document}